\documentclass[10pt]{article}
\usepackage{amsmath,amsthm,amssymb,mathrsfs}
\usepackage{enumerate}
\usepackage{graphics,tikz-cd}
\usepackage{float}
\usepackage[T1]{fontenc}
\usepackage[textwidth=3.4cm,textsize=small]{todonotes}
\usepackage[normalem]{ulem}
\usepackage{enumerate}
\usepackage{graphics,tikz-cd}
\usepackage{tikz,tkz-euclide}
\usepackage{multicol}
\usepackage[width=14cm]{geometry}
\usetikzlibrary{arrows,calc,patterns}

\usepackage[numbers,sort&compress]{natbib}
\bibpunct[, ]{[}{]}{,}{n}{,}{,}

\renewcommand\ge\geqslant
\renewcommand\geq\geqslant
\renewcommand\le\leqslant
\renewcommand\leq\leqslant

\numberwithin{equation}{section}

\newcommand{\refpart}[1]{{\it (#1)}}  

\newcommand{\PP}{\mathbb{P}}

\newcommand{\CC}{\mathbb{C}}

\newcommand{\RR}{\mathbb{R}}

\newcommand{\DD}{{\mathcal D}}
\newcommand{\DI}{{\mathcal I}}
\newcommand{\DR}{{\mathcal R}}

\newtheorem{theorem}{Theorem}[section]
\newtheorem{lemma}[theorem]{Lemma} 
\newtheorem{propose}[theorem]{Proposition}

\newtheorem{remark}[theorem]{Remark}  
\newtheorem{example}[theorem]{Example}

\newcommand{\pcoor}[1]{%
\begingroup\lccode`~=`: \lowercase{\endgroup
\edef~}{\mathbin{\mathchar\the\mathcode`:}\nobreak}%
(
\begingroup
\mathcode`:=\string"8000
#1%
\endgroup 
)
}

\title{\bf Dupin Cyclides as a Subspace of Darboux Cyclides}

\author{
Jean Michel Menjanahary\; and\; Raimundas Vidunas \\
\em Faculty of Mathematics and Informatics, Vilnius University, Lithuania
}
\date{\empty}

\begin{document}
\maketitle

\begin{abstract}
	Dupin cyclides are interesting algebraic surfaces used in geometric design and architecture
	to join canal surfaces smoothly and to construct model surfaces. 
	Dupin cyclides are special cases of Darboux cyclides, which in turn are rather general surfaces in $\RR^3$
	of degree 3 or 4. This article derives the algebraic conditions %
	for recognition of Dupin cyclides among the general implicit form of Darboux cyclides.
	We aim at practicable sets of algebraic equations on the coefficients of the implicit equation,
	each such set defining a complete intersection (of codimension 4) locally. 
	Additionally, the article classifies all real surfaces and lower dimensional degenerations 
	defined by the implicit equation for Dupin cyclides.
\end{abstract}


\section{Introduction}
{\em Darboux cyclides} are classical algebraic surfaces in $\RR^3$ that 
have promising applications in geometric design and architecture. Their implicit equation has the form
\begin{align} \label{eq:gendarb}
	a_0\big(x^2+y^2+z^2\big)^2&+2(b_1x+b_2y+b_3z)\big(x^2+y^2+z^2\big)\nonumber\\
	&+c_1x^2+c_2y^2+c_3z^2+2d_1yz+2d_2xz+2d_3xy\\
	&+2e_1x+2e_{2}y+2e_{3}z+f_{0} =  0.\nonumber
\end{align}
Here $a_0,b_1,\ldots,f_0$ are real coefficients. 
Remarkably, Darboux cyclides are covered by several families of circles  \cite{Darboux,DarCyc,Kra2}. 
Hence they are natural candidates to model a surface composed of patches blended along circles. 
This task is still challenging for more general Daboux cyclides \cite{Zhao},
but the special case of Darboux cyclides called {\em Dupin cyclides} have definite applications already 
\cite{Martin, DuPratt1, DuPratt2,Druoton, Blend1, Blend2, Zube, Kra}. 
Dupin cyclides are canal surfaces 
whose curvature lines are circles or lines.
They are useful to join pipes between canal surfaces \cite{Druoton} 
and to model surfaces smoothly blended along curvature lines \cite{Blend1, Blend2, Zube, Kra}. 
Significantly, the set of Dupin cyclides is stable under the offsetting at a fixed distance along 
the surface normals \cite{DuPratt1,Peternell,Zube}. 
The offset operation arises frequently in geometric design and manufacturing \cite{DarCyc,DuPratt1}. 
On that account, geometric modeling 
with Dupin cyclides simplifies 
computation of offset surfaces.

It is evidently desirable to distinguish Dupin cyclides among general Darboux cyclides. 
The implied standard recognition procedures involve bringing the implicit equation (\ref{eq:gendarb})
to a known canonical form by M\"obius transformations \cite{Zhao,Bastl14},
or discerning that a geometric characterization is satisfied \cite{DuChand,Kra,Ottens}.
To establish 
a more convenient recognition procedure, we compute the set of algebraic equations 
on the coefficients $a_0,b_1,\dots,f_0$ characterizing Dupin cyclides among the form (\ref{eq:gendarb}).
Our starting point is the following canonical forms of Dupin cyclides under the Euclidean transformations.
A {quartic Dupin cyclide} can be presented by the implicit equation
\begin{equation}\label{eq:dupin} 
	\big(x^2+y^2+z^2+\alpha^2-\gamma^2-\delta^2\big)^2-4(\alpha x-\gamma\delta)^2-4(\alpha^2-\gamma^2)y^2=0,
\end{equation}
after Euclidean translations and rotations \cite[p.223--224]{DuPratt1}. 
We broadly assume that $\alpha^2$, $\gamma^2$, $\delta^2$, $\alpha\gamma\delta\in \RR$,
thereby allowing cyclides without real points along other degenerate cases. 
All degenerate cases are described further in Section \ref{sc:degenerate}.
A {cubic Dupin cyclide} can be brought to the form
\begin{align}\label{eq:canpara}
	2x(x^2+y^2+z^2) - (p +q )x^2 - p y^2-q z^2+\frac{p q }{2}x=0
\end{align}
with $p,q \in \RR$. This differs from the equation on \cite[p.151]{DuPratt2} by scaling of $p,q$ with factor 2.
M\"obius transformations on $\RR^3\cup\{\infty\}$ may convert any Dupin cyclide to a torus,
as we recall in Section \ref{sec:izomobius}.

It is straightforward to normalize the coefficients $b_1,b_2,b_3$ to zero (if \mbox{$a_0\neq 0$}) by Euclidean translations,
but further normalization by orthogonal or M\"obius transformations is cumbersome.
This article presents the computed set of necessary and (generically) sufficient algebraic conditions on $a_0,b_1,\ldots,f_0$ 
so that the equation (\ref{eq:gendarb}) defines a Dupin cyclide. 
The 14 coefficients are viewed as the homogeneous coordinates $(a_0:b_1:\ldots:f_0)$ in the real projective space $\PP^{13}$,
which is identified as the space of Darboux cyclides.
The Dupin cyclides are represented by the projective variety $\DD_0$ in $\PP^{13}$ defined by the found algebraic conditions.
Some of the points on this variety represent degenerations of cyclides to reducible, non-reduced or quadratic surfaces.

To organize the results, the cases of quartic and cubic cyclides are considered separately. 
The subvariety of $\DD_0$ representing quartic Dupin cyclides has $a_0\neq 0$ in (\ref{eq:gendarb}); it will be denoted by $\DD_4$. 
The subvariety of $\DD_0$ representing cubic Dupin cyclides (i.e., 
with $a_0=0$ and $b_1^2+b_2^2+b_3^2 \neq 0$) will be denoted by $\DD_3$.
We are interested only in the real points on those varieties so that the coefficients in (\ref{eq:gendarb}) are real. 
The next section states the main results of this article: 
the algebraic equations that characterize the two main subvarieties $\DD_4$ and $\DD_3$. 
The results are proved in Section \ref{sec:3} (for quartic cyclides) and Section \ref{sec:4} (for cubic cyclides). 


As summarized in Section \ref{sc:climit1}, the co-dimension of the considered spaces of Dupin cyclides 
inside the respective projective spaces of Darboux cyclides equals 4.
In particular, the variety $\DD_0$ 
has dimension 9. 
The main explicit result is presented in Theorem \ref{th:m1}
by underscoring open subvarieties of those varieties that are {\em complete intersections}
in a suitable ambient open subspace of $\PP^{13}$.  
This way of presenting the results should be convenient for practical applications, we suggest.
The results are applied in Section \ref{sc:qrr} to compute an important invariant of Dupin cyclides under the M\"obius transformations.
Additionally, Section \ref{sec:class} classifies the real surfaces defined by the equations for Dupin cyclides, 
including degenerations to a few or no real points. 

\section{The main results}
\label{sec:2}

To present the results in more compact form, these abbreviations are used throughout the article:
\begin{align} \label{eq:r0}
	B_0 = & \; b_1^2+b_2^2+b_3^2,\\
	C_0 = & \; c_1+c_2+c_3,\\ 
	E_0 = & \; e_1^2+e_2^2+e_3^2,\\
	W_1 = &\;  c_1c_2+c_1c_3+c_2c_3-d_1^2-d_2^2-d_3^2,\\
	W_2 = &\; c_1c_2c_3+2d_1d_2d_3-c_1d_1^2-c_2d_2^2-c_3d_3^2, \\
	W_3 = &\; b_1^2c_1 + b_2^2c_2 + b_3^2c_3 + 2b_2b_3d_1 + 2b_1b_3d_2 + 2b_1b_2d_3,\\
	\label{eq:r4}
	W_4 = &\;c_1e_1^2+c_2e_2^2+c_3e_3^2+2d_{1} e_{2} e_{3} + 2d_{2} e_{1} e_{3}  + 2 d_{3} e_{1} e_{2}.
\end{align}
These expressions are symmetric under the permutations of the variables $x,y,z$, 
or equivalently, under the permutations of the indices $1,2,3$. 
We will use several non-symmetric expressions, starting from
\begin{align}\label{eq:gen1}
	K_1 = (c_3-c_2)e_2e_3+d_1(e_2^2-e_3^2)+(d_2e_2-d_3e_3)e_1.
\end{align}
Let $\sigma_{12}$, $\sigma_{13}$, $\sigma_{23}$ be the permutations of the coefficients in (\ref{eq:gendarb})
which permute the indices $1,2$ or $1,3$ or $2,3$, respectively. This allows us to express 
variations of non-symmetric expressions  straightforwardly. In particular,
\begin{align} \label{eq:gen1sa}
	\sigma_{12}K_1 = & \; (c_3-c_1)e_1e_3+d_2(e_1^2-e_3^2)+(d_1e_1-d_3e_3)e_2, \\  \label{eq:gen1sb}
	\sigma_{13}K_1 = & \; (c_1-c_2)e_1e_2+d_3(e_2^2-e_1^2)+(d_2e_2-d_1e_1)e_3.
\end{align}

\subsection{Recognition of quartic Dupin cyclides}

In order to simplify recognition of quartic Dupin cyclides among Darboux cyclides,
we first assume $a_0=1$ in $(\ref{eq:gendarb})$ without loss of generality.
Thereby the ambient space of Darboux cyclides is identified with the affine space $\RR^{13}$ rather than $\PP^{13}$.
Further, we can easily apply the shift
\begin{equation} \textstyle  \label{eq:transl}
	(x,y,z) \mapsto \left( x-\frac12\,b_1, y-\frac12\,b_2, z-\frac12\,b_3\right)
\end{equation}
and eliminate the cubic term $2(b_1x+b_2y+b_3z)\big(x^2+y^2+z^2\big)$. 
Thus the recognition problem simplifies to consideration of cyclides of the form
\begin{align} \label{eq:mainForm}
	\big(x^2+y^2+z^2\big)^2
	&+c_1x^2+c_2y^2+c_3z^2+2d_1yz+2d_2xz+2d_3xy\nonumber\\
	&+2e_1x+2e_{2}y+2e_{3}z+f_{0}  =  0.
\end{align}
One could further apply orthogonal or inversion transformations to bring the quartic equation 
to an even simpler canonical form with $d_1=d_2=d_3=0$, 
but those transformations are cumbersome to calculate. 
Recognition of Dupin cyclides in the form (\ref{eq:mainForm}) is therefore a pivotal practical problem.
The ambient space of Darboux cyclides simplifies accordingly to a 10-dimensional affine space $\RR^{10}$
with the coordinates $c_1,c_2,\ldots,f_0$. We denote by $\DD_4^{*}$ the variety of Dupin cyclides there. 

The variety $\DD_4$ is the orbit of $\DD_4^{*}$ under the easy translations (\ref{eq:transl}). 
The following theorem describes the equations for $\DD_4^*$. The equations for $\DD_4$ are obtained by a straightforward modification of the coefficients in (\ref{eq:mainForm}), as described in Section \ref{sec:whole}.
Beside (\ref{eq:gen1}), we immediately use these polynomials:
\begin{align}\label{eq:gen2}
	L_1 = & \;
	\big(W_1+4f_0-(c_2+c_3)^2-d_2^2-d_3^2\big)e_1  \\
	& \, +\big(C_0d_3+c_3d_3-d_1d_2\big)e_2 +\big(C_0d_2+c_2d_2-d_1d_3\big)e_3, \nonumber  \\
	M_1 = &\;  2(c_1e_1+d_3e_2+d_2e_3)(W_1+4f_0)+e_1(W_2-C_0W_1-4E_0).\label{eq:gen9}
\end{align}
\begin{theorem}\label{th:m1}
	The hypersurface in $\RR^3$ defined by $(\ref{eq:mainForm})$ is a Dupin cyclide 
	if and only if one of the following cases holds:
	\begin{enumerate}[(a)]
		\item $e_1\neq 0$, $\sigma_{12} K_1=0$, $\sigma_{13} K_1=0$, $L_1=0$, $M_1=0$.
		\vspace{2pt}
		\item $e_2\neq 0$, $K_1=0$, $\sigma_{13} K_1=0$, $\sigma_{12} L_1=0$, $\sigma_{12} M_1=0$.
		\vspace{2pt}
		\item $e_3\neq 0$, $K_1=0$, $\sigma_{12} K_1=0$, $\sigma_{13} L_1=0$, $\sigma_{13} M_1=0$.
		\vspace{1pt}
		\item $e_1=e_2=e_3=0$, $W_1+4f_0=0$, $W_2-C_0W_1=0$.
		\item $e_1=e_2=e_3=0$, $C_0 \neq 0$, $\big(4W_1+12f_0-C_0^2\big)^2-16f_0C_0^2=0$,\\[1pt] 
		$(4W_1+12f_0-3C_0^2)(W_1+4f_0)-2C_0(W_2-C_0W_1)=0$.
		\vspace{2pt}
		\item  $e_1=e_2=e_3=0$, $C_0=0$, $W_1+3f_0=0$, $(W_2-C_0W_1)^2-4f_0^3=0$.
	\end{enumerate}
\end{theorem}

\begin{example}\rm  \label{rm:torus}
	A prototypical example of a Dupin cyclide is the torus with the minor radius $r$ and the major radius $R$. 
	It is defined by the equation
	\begin{align} \label{eq:torus}
		\big(x^2+y^2+z^2+R^2-r^2\big)^2-4 R^2 ( x^2+y^2)=0.
	\end{align}
	Our main theorem applies with $e_1=e_2=e_3=0$ and
	\[
	c_1=c_2=-2R^2-2r^2,\ c_3=2R^2-2r^2,\ d_1=d_2=d_3=0,\ f_0=(R^2-r^2)^2,
	\]
	Case \refpart{e} applies, as $C_0=-2R^2-6r^2<0$, and
	its last two equalities 
	hold with $W_1=4(R^2+r^2)(3r^2-R^2)$ and $W_2=8(R^2+r^2)^2(R^2-r^2)$.
\end{example}

\begin{remark} \rm \label{rm:main}
	The cases of Theorem \ref{th:m1} 
	define a stratification of the variety $\DD_4^{*}$ into pieces
	that are complete intersections in $\RR^{10}$, possibly of variable dimension. 
	This localization onto complete intersections is our deliberate strategy of presenting a practical
	procedure of recognizing Dupin cyclides. The aim is to check the minimal number 
	of (rather cumbersome) equations for each particular cyclide.
	The co-dimension of $\DD_4^*$ in $\RR^{10}$ turns out to be $4$, 
	hence this minimal number of equations equals $4$.
	
	Concretely, the parts \refpart{a}--\refpart{c} define 3 intersecting open subvarieties 
	of $\DD_4^{*}$ as complete intersections on the Zariski open subsets $e_1\neq 0$, $e_2\neq 0$ and  $e_3\neq 0$ of $\RR^{10}$.
	Only 4 equations are checked in these cases, as the codimension equals 4.
	The cases \refpart{d}--\refpart{f} define subvarieties of $\DD_4^{*}$ of smaller dimensions 
	inside the closed subset $e_1=e_2=e_3=0$ of $\RR^{10}$. 
	There we have two reduced components \refpart{d},  \refpart{e} 
	of dimension 5, and the former is a complete intersection 
	already. The latter component is further stratified into the cases $C_0\neq 0$ and $C_0=0$,
	leading to the concluding complete intersections \refpart{e}, \refpart{f}  of the codimension 5 or 6, respectively.
\end{remark}

\subsection{Recognition of cubic Dupin cyclides}

The general cubic Darboux cyclides have implicit equation of the form
\begin{align} \label{eq:gendarbp}
	& \, 2(b_1x+b_2y+b_3z)\big(x^2+y^2+z^2\big) \nonumber \\
	& +c_1x^2+c_2y^2+c_3z^2+2d_1yz+2d_2xz+2d_3xy \nonumber\\
	& \hspace{89pt} +2e_1x+2e_{2}y+2e_{3}z+f_{0}  =  0.
\end{align}
The ambient space of Darboux cyclides is therefore considered as the real projective space $\PP^{12}$, in which we describe $\DD_3$. To formulate the result for the cubic cyclides, we define the rational expression
\begin{align}
	E_1 = 
	&\;  -\frac{b_1}{B_0} \! \left(\frac{W_3}{B_0}-c_2-c_3\right)^{\!2} 
	+ \frac{2b_1^2}{B_0^2}\,(b_3c_3d_2 + b_2c_2d_3)
	- \frac{4b_1}{B_0^2}\,(b_3d_2 + b_2d_3)^2 \nonumber \\
	&\; + \frac{2(b_3d_2 + b_2d_3)}{B_0^2}\,(b_2^2c_1 + b_3^2c_1-2b_2b_3d_1)
	- \frac{2b_2b_3}{B_0^2}\,(c_2 - c_3)(b_2d_2 - b_3d_3)\nonumber\\
	&\;+\frac{b_1}{B_0}\,\big((c_1 - c_2)(c_1 - c_3)- d_1^2 + d_2^2 + d_3^2\big)
	+\frac{2d_1}{B_0}\,(b_2d_2 + b_3d_3).
\end{align}
\begin{theorem}\label{th:m2}
	The hypersurface in $\RR^3$ defined by $(\ref{eq:gendarbp})$ is a Dupin cyclide if and only if 
	\begin{align}
		e_1= &\, \textstyle
		\frac{1}{4}\,E_1, \qquad e_2=\frac{1}{4}\,\sigma_{12} E_1, \qquad e_3=\frac{1}{4}\,\sigma_{13} E_1,\\[2pt]
		f_0= &\, \frac{W_3}{4B_0^2}\left(\frac{W_3}{B_0}-C_0\right)^{\!2} +\frac{W_3W_1}{4B_0^2}+ \frac{W_2-C_0W_1}{4B_0}.
	\end{align}
\end{theorem}
The co-dimension of $\DD_3$ equals 4, and the dimension 
equals 8 within the hyperplane $\PP^{12}\subset \PP^{13}$.
With $B_0\neq 0$, the coefficients $b_1,b_2,\ldots,c_1,\ldots,d_3$ to the cubic and quadratic parts can be chosen freely, 
and then there are unique values for $e_1,e_2,e_3,f_0$ so that (\ref{eq:gendarbp}) defines a Dupin cyclide.
The analogous question for quartic cyclides is considered in Remark \ref{rem:(6:1)}. 

\section{Quartic Dupin cyclides}
\label{sec:3}

In this section we prove Theorem \ref{th:m1} for recognition of quartic Dupin cyclides of the form (\ref{eq:mainForm}). 
The proof refers to Gr\"obner basis computations which were done using computer algebra packages 
{\sf Maple} and  {\sf Singular}. 
But 
we also present constructive ways of obtaining the presented equations from the initial ones.
The initial equations are derived from the well-known canonical form (\ref{eq:dupin}) 
of quartic Dupin cyclides. 
We consider the variety $\DD_4^*$ as the orbit of this canonical form under the orthogonal transformations $O(3)$.
Rather than introducing the orthogonal transformations explicitly and eliminating their parameters,
we compare the invariants under $O(3)$ for the general equation (\ref{eq:mainForm}) and the canonical equation.
This effective comparison is done in Section \ref{sec:ortho4}. 
The coefficients of the canonical form are eliminated in Section \ref{sec:witha}.
Finally, Section \ref{sec:mproof} finds the complete intersection cases of Theorem \ref{th:m1} 
as expounded in Remark \ref{rm:main}.

\subsection{From the canonical form}

We adopt the parametrized description of the quartic equation (\ref{eq:dupin}) for quartic Dupin cyclides 
to the implicit form like (\ref{eq:mainForm}). 
\begin{lemma}
	A quartic Dupin cyclide can be expressed, up to translations and orthogonal transformations in $\RR^3$, 
	to the form
	\begin{align} \label{eq:orthdupin}
		\big(x^2+y^2+z^2\big)^2
		+A_1x^2+A_2y^2+A_3z^2+Dx+F =  0,
	\end{align}
	with the relations
	\begin{align} \label{eq:starta}
		D^2= &\, -(A_2+A_3)(A_1-A_2)(A_1-A_3), \\
		\label{eq:startb}
		4F= &\; A_2^2+A_3^2+A_2A_3-A_1A_2-A_1A_3.
	\end{align}
\end{lemma}
\begin{proof}
	The comparison of (\ref{eq:dupin}) and (\ref{eq:orthdupin}) gives these relations
	\begin{align} 
		& A_1 =  -2(\alpha^2+\gamma^2+\delta^2), \quad
		A_2 = 2(\gamma^2-\alpha^2-\delta^2), \quad
		A_3 =  2(\alpha^2-\gamma^2-\delta^2),  \nonumber\\
		& D =  8\alpha\gamma\delta, \qquad
		F = (\alpha^2-\gamma^2-\delta^2)^2-4\gamma^2\delta^2.
	\end{align}
	We eliminate $\alpha,\gamma,\delta$, and obtain (\ref{eq:starta})--(\ref{eq:startb}).
	Necessity 
	of these relations follows from the fact that there are no non-trivial $O(3)$-symmetries of equation (\ref{eq:orthdupin}).
\end{proof}

\noindent
The canonical form defines a variety of dimension $3=5-2$, as we have 5 coefficients in (\ref{eq:orthdupin}) 
and 2 relations between them. The $O(3)$ action 
adds 3 degrees of freedom, hence the dimension in $\RR^{10}$ has to equal 6.

\begin{remark} \rm
	An inverse map is defined by 
	\begin{equation} \label{eq:greeksq} 
		\alpha=\frac{\sqrt{A_3-A_1}}2, \qquad
		\gamma=\frac{\sqrt{A_2-A_1}}2, \qquad
		\delta=\frac{\sqrt{-A_2-A_3}}2.
	\end{equation}
	Each of these values can be multiplied by $-1$, as  long as $D=8\alpha\gamma\delta$. 
\end{remark}

\begin{remark} \rm
	The cases of Theorem \ref{th:m1}  
	with $e_1=e_2=e_3=0$ are in the orbit of the canonical form (\ref{eq:orthdupin}) with $D=0$. 
	The splitting into the cases \refpart{d} and \refpart{e}--\refpart{f} is consistent with the expression $D=\alpha\gamma\delta$. 
	The canonical form for the case \refpart{d} has $\delta=0$ in (\ref{eq:dupin}), or $A_2+A_3=0$, $D=0$, $F=\frac14A_2^2$ in (\ref{eq:orthdupin}). 
	The canonical form for the case 
	\refpart{e} has either $\alpha=0$, $A_1=A_3$, $F=\frac14A_2^2$, or  $\gamma=0$, $A_1=A_2$, $F=\frac14A_3^2$.
	The canonical form for the case 
	\refpart{f} has more particularly $A_1=A_3$, $A_2=-2A_1$ (or $A_3=-2A_1$), and $F=A_1^2$.
\end{remark}

\subsection{Applying orthogonal transformations}
\label{sec:ortho4}

The direct way to compute the $O(3)$-orbit of the canonical form (\ref{eq:dupin}) 
is  to apply an arbitrary orthogonal $3\times 3$ matrix to the vector $(x,y,z)$ of the indeterminates.
The coefficients would be then parametrized by the 14 variables --- the 5 coefficients in (\ref{eq:orthdupin}), and
the 9 entries of the $3\times 3$ matrix --- restrained by two equations (\ref{eq:starta})--(\ref{eq:startb}) and 
the $6$ orthonormality conditions between the rows on the $3\times 3$ matrix.
The expected dimension of $\DD_4^*$ is thereby confirmed: $6=14-2-6$.
But elimination of the parametrizing variables appears to be too cumbersome even using computer algebra 
systems such as {\sf Maple} and {\sf Singular}. 

Instead of working with the 9 variables of the orthogonal matrix, we identify the $O(3)$-invariants 
for the equations (\ref{eq:mainForm}) and (\ref{eq:orthdupin}).
The group $O(3)$ 
acts on the quadratic and linear parts of these equations disjointly, 
making identification of invariants and their relations easier.
The further elimination of $A_1,A_2,A_3,D,F$ is done in the next section.
\begin{lemma} \label{th:witha2a3}
	\label{thm:orth_relation}
	The hypersurfaces $(\ref{eq:orthdupin})$ and $(\ref{eq:mainForm})$ are related by an orthogonal transformation on $(x,y,z)$
	if and only if these relations hold:
	\begin{align}  \label{eq:chp1}
		A_1+A_2+A_3 = &\; c_1+c_2+c_3,\\ 
		\label{eq:chp2}
		A_1A_2+A_1A_3+A_2A_3= &\; c_1c_2+c_1c_3+c_2c_3-d_1^2-d_2^2-d_3^2,\\ 
		\label{eq:chp3}
		A_1A_2A_3 = &\; c_1c_2c_3+2d_1d_2d_3-c_1d_1^2-c_2d_2^2-c_3d_3^2, \\  
		\label{eq:eigenv1}
		A_1e_1 = &\;  c_1e_1+d_3e_2+d_2e_3,\\  \label{eq:eigenv2}
		A_1e_2= &\; d_3e_1+c_2e_2+d_1e_3,\\  \label{eq:eigenv3}
		A_1e_3 = &\; d_2e_1+d_1e_2+c_3e_3,\\  \label{eq:ee}
		D^2= &\; 4E_0, \\  \label{eq:ff}
		F= &\; f_0.
	\end{align}
\end{lemma}
\begin{proof}
	An orthogonal transformation acts as follows:
	\begin{itemize}
		\item The highest degree term $(x^2+y^2+z^2)^2$ remains invariant.
		\item The quadratic forms  $c_1x^2+c_2y^2+c_3z^2+2d_1yz+2d_2xz+2d_3xy$
		and $A_1x^2+A_2y^2+A_3z^2$ are related by a conjugation between their symmetric matrices
		\begin{equation} \label{eq:qfmatrices}
			P=\left( \begin{array}{ccc} c_1 & d_3 & d_2 \\ d_3 & c_2 & d_1 \\ d_2 & d_1 & c_3 
			\end{array} \right)
			\qquad \mbox{and} \qquad
			Q=\left( \begin{array}{ccc} A_1 & 0 & 0 \\ 0 & A_2 & 0 \\ 0 & 0 & A_3 
			\end{array} \right).
		\end{equation}
		As is well known, quadratic forms are transformed by the corresponding matrix transformations $P\mapsto M^TPM$.
		Orthogonal matrices satisfy \mbox{$M^T=M^{-1}$,} hence $O(3)$ 
		conjugates the matrix $P$.
		Accordingly, we can compare the characteristic polynomials and obtain
		(\ref{eq:chp1})--(\ref{eq:chp3}).
		\item The linear forms $2e_1x+2e_{2}y+2e_{3}z$ and $Dx$ are related by 
		the same orthogonal transformation $M$ acting on the corresponding vectors 
		$(2e_1,2e_{2},2e_{3})$ and $(D,0,0)$. Their relation to the matrices in (\ref{eq:qfmatrices})
		will be preserved, thus $(e_1,e_{2},e_{3})$ must be an eigenvector of the first matrix
		with the eigenvalue $A_1$. This gives the relations (\ref{eq:eigenv1})--(\ref{eq:eigenv3}).
		Besides, the Euclidean norms of the two vectors will be equal, giving (\ref{eq:ee}).
		\item The constant coefficients will be equal, giving (\ref{eq:ff}). 
	\end{itemize}
\end{proof}

\subsection{Elimination of the coefficients of the canonical form}
\label{sec:witha}

Here we start using the abbreviations (\ref{eq:r0})--(\ref{eq:r4}) and the algebraic language of ideals. 
Let us denote the polynomial ring 
\begin{equation}
	\DR_4^{*}= \RR[c_1,c_2,c_3,d_1,d_2,d_3,e_1,e_2,e_{3},f_{0}].
\end{equation}
The variety $\DD_4^{*}$ 
of Dupin cyclides is defined by the ideal $\DI_4^{*}\subset \DR_4^{*}$ obtained by eliminating $A_1,A_2,A_3,D,F$ 
from the equations (\ref{eq:starta})--(\ref{eq:startb}) and (\ref{eq:chp1})--(\ref{eq:ff}).
As an intermediate step, it is straightforward to eliminate $A_2,A_3,D,F$
and leave only $A_1$ as an auxiliary variable.
\begin{lemma}  \label{th:dgraph}
	The hypersurface $(\ref{eq:mainForm})$ is a Dupin cyclide  
	if and only if there exists $A_1\in\RR$ such that these polynomials vanish:
	\begin{align} 
		G_1 = &\;  -e_1A_1+c_1e_1+d_3e_2+d_2e_3,\label{eq:eigenva1}\\  
		G_2= &\; -e_2A_1+d_3e_1+c_2e_2+d_1e_3,\label{eq:eigenva2}\\  
		G_3 = &\; -e_3A_1+d_2e_1+d_1e_2+c_3e_3,\label{eq:eigenva3}\\ 
		H_1 = & \;  2\,(W_1+4f_0)A_1  + W_2-C_0W_1 - 4 E_0, \label{eq:eqqf1}\\
		H_2 = & \; (C_0^2+4W_1+12f_0)A_1 - C_0^3 +4C_0f_0-4 E_0,\label{eq:eqqf3}\\
		H_3 = & \; A_1^2-2C_0A_1 + C_0^2 - W_1 - 4f_0.\label{eq:eqqf4} 
	\end{align}
\end{lemma}
\begin{proof} 
	The given polynomials generate the same ideal in $\DR_4^{*}[A_1]$ 
	as the ideal obtained after an elimination of $A_2,A_3,D,F$ from equations (\ref{eq:chp1}--\ref{eq:ff}). 
	This can be checked by computing and comparing reduced Gr\"obner bases.
\end{proof}

Here is an explicit reversible transformation between the equations of Lemmas \ref{th:witha2a3} and \ref{th:dgraph}.
Equations (\ref{eq:eigenv1})--(\ref{eq:eigenv3}) do not contain the variables $A_2,A_3,D,F$ we eliminate,
so they are copied as $G_1=G_2=G_3=0$.
The other equations are symmetric in $A_2,A_3$. Using (\ref{eq:chp1}), (\ref{eq:chp2}),
equation (\ref{eq:chp3}) becomes $H^*_1=0$ with
\begin{align} \label{eq:charpa}
	H^*_1 = &\; A_1^3-C_0A_1^2+W_1A_1 - W_2.
\end{align}
This is the characteristic polynomial of the first matrix in (\ref{eq:qfmatrices}), of course.
Elimination of $A_2,A_3,D$ from (\ref{eq:starta})--(\ref{eq:startb}) gives
\begin{align*}
	& 4E_0= (A_1-C_0)(W_1-2C_0A_1+3A_1^2).
\end{align*}
We expand this equation to $H^*_2=0$, where 
\begin{align} \label{eq:eq3}
	H^*_2= & \; 3A_1^3-5C_0A_1^2+\big(2C_0^2+W_1\big)A_1-C_0W_1-4E_0.
\end{align}
Equation (\ref{eq:startb}) with eliminated $A_2,A_3,F$ becomes $H_3=0$.
Considering $H^*_1,H^*_2,H_3$ 
as polynomials in $A_1$,  we divide the two cubic 
$H^*_1,H^*_2$ by the quadratic 
$H_3$. The division remainders
\begin{align}
	\widehat{H}_1= & \; H^*_1-(A_1+C_0)H_3, \\
	\widehat{H}_2 = & \; H^*_2-(3A_1+C_0)H_3, \nonumber 
\end{align}
are linear in $A_1$. Indeed,  $\widehat{H}_2=H_2$ and 
\begin{align} \label{eq:eqq1}
	\widehat{H}_1 = &  \; (C_0^2+2W_1+4f_0)A_1  - C_0(C_0^2-W_1-4f_0) - D_0.
\end{align}
We modify $\widehat{H}_1$ to the somewhat simpler $H_1=H_2-\widehat{H}_1$.

Elimination of $A$ from the six equations (\ref{eq:eigenva1})--(\ref{eq:eqqf4}) 
is not complicated, as five of them are linear in $A$. 
\begin{propose} \label{th:d0gen}
	The ideal $\DI_4^*$ specifying Dupin cyclides in $(\ref{eq:mainForm})$ is generated by 
	the following $12$ polynomials:
	\begin{enumerate} 
		\item $K_1$, $K_2=\sigma_{12} K_1$, $K_3=\sigma_{13}K_1$; see $(\ref{eq:gen1})$--$(\ref{eq:gen1sb})$;
		\vspace{3pt}
		\item $L_1$, $L_2=\sigma_{12} L_1$, $L_3=\sigma_{13}L_1$; see $(\ref{eq:gen2})$;
		\vspace{3pt}
		\item $M_1$, $M_2=\sigma_{12} M_1$, $M_3=\sigma_{13}M_1$; see $(\ref{eq:gen9})$;
		\vspace{3pt}
		\item $N_1 = \big(4W_1+12f_0-3 C_0^2\big)(W_1+4f_0) - 2C_0(W_2-C_0W_1-6E_0)-4\,W_4$, \\
		$N_2 = 4(W_2-C_0W_1-2E_0)(W_1+4f_0) +\big(C_0^2-4f_0\big) \big(W_2+C_0W_1+8C_0f_0-4E_0\big)$, \\
		$N_3 =  \big(\, W_2+C_0W_1+8C_0f_0-4E_0\big)^{\!2}-4(W_1+4f_0)^3$.
	\end{enumerate} 
\end{propose}
\begin{proof}  The ideal $\DI_4^*$ is obtained by eliminating $A_1$ from the polynomials 
	\mbox{(\ref{eq:eigenva1})--(\ref{eq:eqqf3}).} 
	Gr\"obner basis comparison shows that the ideal in $\DR_4^*$ 
	generated by the listed 12 polynomials coincides with $\DI_4^*$. 
\end{proof}

The 12 polynomials of this proposition can be derived explicitly from Lemma \ref{th:dgraph} as follows. 
Most straightforwardly, $K_1,K_2,K_3$ are obtained by eliminating $A_1$ 
from the pairs of polynomials in (\ref{eq:eigenv1})--(\ref{eq:eigenv3}).
The polynomial $L_1$ turns up 
as follows: 
\begin{align}
	L_1 = & \; -e_1H_3+(c_1+2c_2+2c_3-A_1)G_1-d_3G_2-d_2G_3. 
\end{align}
The polynomials $L_2=\sigma_{12}L_1$, $L_3=\sigma_{13}L_1$ are obtained similarly.
Further, $M_1$, $M_2$, $M_3$ are obtained by pairing $H_1$ with $G_1$, $G_2$ or $G_3$, and 
eliminating $A_1$. The polynomial $N_1$ is obtained by the combination
\begin{align*}
	N_1= \, -(C_0^2+4W_1+12f_0)H_3+A_1H_2-C_0(H_2+2H_1)-4e_1G_1-4e_2G_2-4e_3G_3.
\end{align*}
It is clear that $N_2$ is the resultant of $H_1$ and $H_2$ with respect to $A_1$. 
The polynomial $N_3$ is the resultant of $H_1$ and $H_3$ with respect to $A_1$. 
Its compact expression is obtained by translating $A_1=\widetilde{A}_1+C_0$ 
so that 
\begin{align*}
	H_3&=\widetilde{A}_1^2-W_1-4f_0,\\
	H_1&=2(\widetilde{A}_1+C_0)(W_1+4f_0)+W_2-C_0W_1-4E_0,
\end{align*}
and by computing the resultant as the determinant of the Sylvester matrix
\[
\left( \begin{array}{ccc}  1 & 0 & -W_1-4f_0 \\
	2W_1+8f_0 & W_2+C_0W_1+8C_0f_0-4E_0  & 0 \\
	0 & 2W_1+8f_0 & W_2+C_0W_1+8C_0f_0-4E_0 \end{array} \right). 
\]

Compared with Proposition \ref{th:d0gen}, our main Theorem \ref{th:m1} specifies pieces of $\DD_4^*$ 
that are complete intersections in $\RR^{10}$ and cover the whole $\DD_4^*$. 
This is explained in Remark \ref{rm:main}. We wrote {\sf Maple} routines for deciding 
whether a given implicit equation (\ref{eq:gendarb}) defines a Dupin cyclide using either Theorem \ref{th:m1}
or Proposition \ref{th:d0gen}. When we tried to recognize a Dupin cyclide with 5 parameters,
the routine that uses Theorem \ref{th:m1} recognized correctly in a few minutes, 
while the other routine took unreasonably longer.

\subsection{Proof of Theorem \ref{th:m1}}
\label{sec:mproof}

We find convenient complete intersection pieces of $\DD_4^*$ 
by investigating the syzygies between the 12 generators of $\DI_4^*$ in Proposition \ref{th:d0gen}.  
The simplest and most frequent factors of found syzygies 
suggest the localizations in 
$\RR^{10}$ where $\DD_4^*$ requires fewer defining equations.
Localization at those factors shrinks 
the set of generators of $\DI_4^*$. 
In particular, we find that the localizations at $e_1$ (or $e_2$, or $e_3$) 
give complete intersections immediately, leading to the cases \refpart{a}--\refpart{c} of Theorem \ref{th:m1}.
The subvariety $e_1=e_2=e_3=0$ turns out to be reducible.
One component is a complete intersection already, giving the case \refpart{d}. 
The other component is additionally stratified to complete intersections by considering whether $C_0=0$.

Here are some simplest syzygies between the $12$ generators in Proposition \ref{th:d0gen}:
\begin{align}
	0 = & \; e_1K_1+e_2K_2+e_3K_3, \\
	e_1L_2-e_2L_1 = &\; d_2K_1+d_1K_2+(c_1+c_2+2c_3)K_3,\\
	e_3L_1-e_1L_3 = &\; d_3K_1+d_1K_3+(c_1+2c_2+c_3)K_2,\\
	e_2L_3-e_3L_2 = &\; d_3K_2+d_2K_3+(2c_1+c_2+c_3)K_1.
\end{align}
Assume that $e_1\neq 0$. From the first 3 syzygies we see that $K_2=0$, $K_3=0$, $L_1=0$ 
imply $K_1=0$, $L_2=0$, $L_3=0$.
Similarly, we have the syzygy
\begin{align}
	\!2e_1 M_2 -  2e_2M_1= &\; (c_2d_2 - c_3d_2 - 2d_1d_3)K_1 - (2c_2d_1 + 2c_3d_1 + d_2d_3)K_2  \nonumber \\
	&-(2c_3^2 + 2d_1^2 + d_2^2 - 8f_0)K_3 + 2(d_3e_1 - c_1e_2 - c_3e_2 + d_1e_3)L_1 \nonumber\\
	&+(2c_2e_1 + 2c_3e_1 - 2d_3e_2 - d_2e_3)L_2 - d_2e_2 L_3,
\end{align}
and the $\sigma_{23}$-symmetric syzygy with $2e_1M_3 - 2e_3M_1$.
Therefore, if we use $M_1=0$, then we have $M_2=0$, $M_3=0$.
There are similar syzygies that express $e_1N_1$, $e_1N_2$ and $e_1N_3$ in terms of $\DR_4^*$-multiples of lower degree generators as well. Therefore, we obtain the case \refpart{a}. By symmetry between $e_1,e_2$ and $e_3$, specialization at $e_2\neq 0$ gives us the case \refpart{b} and specialization at $e_3\neq 0$ gives us the case \refpart{c}.

Let us consider now the degeneration $e_1=e_2=e_3=0$.
Let $\widehat{\DR}_4^*$ denote the polynomial ring $\RR[c_1,c_2,c_3,d_1,d_2,d_3,f_0]$, 
and let $\varphi:\DR_4^*\to \widehat{\DR}_4^*$ denote the specialization map $e_1=e_2=e_3=0$.
Note that the polynomials $K_1,K_2,K_3,L_1,L_2,L_3,M_1,M_2,M_3$ vanish in $\widehat{\DR}_4^*$ and the image ideal $\varphi(\DI_4^*)$ is generated by $\varphi(N_1),\varphi(N_2)$ and $\varphi(N_3)$. The product $Y_0(W_1+4f_0)$ belongs to the ideal $\varphi(\DI_4^*)$ 
since
\begin{align} 
	Y_0(W_1+4f_0) = (C_0^2+4W_1+12f_0)\,\varphi(N_1)+2C_0\,\varphi(N_2).
\end{align}
If $Y_0\neq 0$, the ideal $\varphi(\DI_4^*)\subset \widehat{\DR}_4^*[Y_0^{-1}]$ is generated by $W_1+4f_0$ and $W_2-C_0W_1$. The option \refpart{d} then follows. Assume that $W_1+4f_0\neq 0$. One can check that the ideal $\varphi(\DI_4^*)$ in $\widehat{\DR}_4^*[(W_1+4f_0)^{-1}]$ 
is generated by $Y_0$, $Y_1$, $Y_2$, $Y_3$, where
\begin{align}
	Y_0 = &\; \big(4W_1+12f_0-C_0^2\big)^2-16f_0C_0^2,\\
	Y_1 = &\; (4W_1+12f_0-3C_0^2)(W_1+4f_0)-2C_0(W_2-C_0W_1),\\
	Y_2 = &\; (W_2-C_0W_1)\big(C_0^2-4W_1-4f_0)-8W_2(W_1+4f_0),\\
	Y_3 = &\; (C_0W_1+9W_2)^2-4W_1^3-4W_2(C_0^3+27W_2).
\end{align} 
Here are two syzygies between them:
\begin{align} 
	-2C_0Y_2 &= 3Y_0(W_1+4f_0)+(C_0^2-12W_1-36f_0)Y_1,\\
	-2C_0Y_3 &= (W_2-C_0W_1-8W_3)(Y_0-3Y_1)-(C_0^2-3W_1)Y_2.
\end{align} 
So the localization with $C_0\neq 0$ gives a complete intersection 
generated by $Y_0$ and $Y_1$, giving us the option \refpart{e}.
If $C_0=0$, then we reduce $Y_0$ to $3f_0+W_1$. After the elimination of $f_0$ with
\begin{equation}
	f_0=\frac13\left(c_1^2-c_2c_3+d_1^2+d_2^2+d_3^2\right)
\end{equation}
we obtain an ideal generated by one element.
We recognize such element compactly as $(W_2-C_0W_1)^2-4f_0^3$ or $(W_2-C_0W_1)^2-4(W_1+4f_0)^3$,
and conclude the last option \refpart{f}. 
It is left to track the case $W_1+4f_0=0$. We have again the syzygy:
\begin{equation} 
	(W_2-C_0W_1)^2=\varphi(N_3)+4\left((W_1+4f_0)^2+C_0W_3-2C_0^2f_0\right)(W_1+4f_0).
\end{equation} 
So we have the ideals inclusion $(W_1+4f_0, W_2-C_0W_1)\subset \varphi(\DI_4^*)+(W_1+4f_0)$. This case is therefore subsumed by \refpart{d}. 

\section{Cubic Dupin cyclides}
\label{sec:4}

Here we prove Theorem  \ref{th:m2}, which characterizes cubic
(also called {\em parabolic} \cite{DuChand}) Dupin cyclides
in the space $\DD_3$ of cubic Darboux cyclides (\ref{eq:gendarb}) with $a_0=0$.
We compute the ideal defining $\DD_3$ as the orbit of a canonical form (\ref{eq:canpara})
of cubic Dupin cyclides under orthogonal transformations and translations.
The $O(3)$-orbit of (\ref{eq:canpara}) 
is computed in  Section \ref{sec:ortho3}, following the same strategy as in Section \ref{sec:ortho4}. 
Theorem \ref{th:m2} is proved in Section \ref{sec:tr3} after applying general translations in $\RR^3$.

\subsection{Applying orthogonal transformations}
\label{sec:ortho3}

Applying an orthogonal transformation to our initial canonical form (\ref{eq:canpara}) 
gives us an intermediate 
form
\begin{align}\label{eq:canorthpara}
	2(b_1x+b_2y+b_3z)\big(x^2+y^2+z^2\big) +\hat c_1x^2+ \hat c_2y^2+\hat c_3z^2 \\
	+2\hat d_1yz+2\hat d_2xz+2\hat d_3xy+2\hat e_1x+2\hat e_2y+2\hat e_3z & =  0. \nonumber
\end{align}
of cubic 
Dupin cyclides. To define the set of generating relations between the coefficients here, 
let us define the polynomials: 
\begin{align} \label{eq:wh1}
	U & = \,  b_1(\hat c_1-\hat c_2-\hat c_3) + 2b_2\hat d_3+2b_3\hat d_2, \\  \label{eq:wh3}
	V &= {\hat c_1}^{\,2}+{\hat c_2}^{\,2}+{\hat c_3}^{\,2}-2\hat c_1\hat c_2-2\hat c_1\hat c_3-2\hat c_2\hat c_3
	+4{\hat d_1}^{\,2}+4{\hat d_2}^{\,2}+4{\hat d_3}^{\,2}.
\end{align}
Recall (\ref{eq:r0}) that we denote $B_0=b_1^2+b_2^2+b_3^2$ .
\begin{lemma}\label{th:orthcubic}
	The cyclide equation $(\ref{eq:canorthpara})$ 
	can be obtained from $(\ref{eq:canpara})$ 
	by an orthogonal transformation on $(x,y,z)$ if and only if these polynomial expressions evaluate to $0$:
	\begin{equation} \label{eq:cbot10}
		B_0-1, \quad U,\quad \sigma_{12}U, \quad\sigma_{13}U, \quad
		16e_1+Vb_1,\quad 16e_2+Vb_2, \quad 16e_3+Vb_3.
	\end{equation}
\end{lemma}
\begin{proof} 
	As in the proof of Lemma \ref{thm:orth_relation}, 
	we consider the $O(3)$ action 
	in each homogeneous part. Clearly, $B_0=1$. 
	The cubic and linear part are proportional:
	\begin{equation} \label{eq:lincu}
		(\hat e_1,\hat e_2,\hat e_3) = \frac{p q }4 \, (b_1,b_2,b_3).
	\end{equation}
	We obtain the following equations from comparison of the quadratic parts and the eigenvector role of $(b_1,b_2,b_3)$:
	\begin{align} \label{eq:ccu1}
		\hat c_1+\hat c_2+\hat c_3&= -2(p +q ), \\  \label{eq:ccu2}
		\hat c_1\hat c_2+\hat c_1\hat c_3+\hat c_2\hat c_3-{\hat d_1}^{\,2}-{\hat d_2}^{\,2}-{\hat d_3}^{\,2} &= (p +q )^2 + p q ,\\
		\label{eq:ccu3} \hat c_1\hat c_2\hat c_3 + 2\hat d_1\hat d_2\hat d_3 
		- \hat c_1{\hat d_1}^{\,2} - \hat c_2{\hat d_2}^{\,2} - \hat c_3{\hat d_3}^{\,2} &= -p q (p +q ),\\
		\label{eq:ccu4} b_1\hat c_1 + b_2\hat d_3 + b_3\hat d_2&=-b_1(p +q ), \\
		\label{eq:ccu5} b_1\hat d_3 + b_2\hat c_2 + b_3\hat d_1&=-b_2(p +q ), \\
		\label{eq:ccu6} b_1\hat d_2 + b_2\hat d_1 + b_3\hat c_3&=-b_3(p +q ).
	\end{align} 
	This system is similar to (\ref{eq:chp1})--(\ref{eq:eigenv3}). Computation and comparison of 
	Gr\"obner bases with respect to the same ordering shows that 
	elimination of $p $, $q $ gives the ideal generated by the polynomials (\ref{eq:cbot10}).
\end{proof}

Constructively, the equations $U=0$, $\sigma_{12}U=0$,  $\sigma_{13}U=0$
are obtained by eliminating $p+q $ in (\ref{eq:ccu1}) and (\ref{eq:ccu4})--(\ref{eq:ccu6}).
From (\ref{eq:ccu2}) we obtain $4p q =-V$. 
The equations for $e_1,e_2,e_3$ then follow from (\ref{eq:lincu}).
The proportionality in (\ref{eq:lincu}) gives the simple equations $b_1\hat e_2=b_2\hat e_1$, 
$b_1\hat e_3=b_3\hat e_1$, $b_2\hat e_3=b_3\hat e_2$.
A Gr\"obner basis with respect to a total degree ordering shows a few more vanishing polynomials of degree 2:
$\ \hat c_1^{\,2}  - (\hat c_2 - \hat c_3)^2 - 4 \hat d_1^{\,2} - 16 b_1\hat e_1$, 
$\ (\hat c_1 - \hat c_2 - \hat c_3)\hat d_1 - 2\hat d_2\hat d_3 + 8b_3\hat e_2$,
$\ (\hat c_1-\hat c_2-\hat c_3)\hat e_1 + 2\hat d_3\hat e_2+2\hat d_2\hat e_3$,
and the $\sigma_{12}/\sigma_{13}$-variants.

\subsection{Proof of Theorem 2.3} 
\label{sec:tr3}

The general form (\ref{eq:gendarbp}) of cubic Darboux cyclides is obtained by applying an arbitrary shift 
\begin{align} \label{eq:tranl3}
(x,y,z)\longmapsto (x+t_1,y+t_2,z+t_3)
\end{align} 
to the form (\ref{eq:canorthpara}), 
up to multiplication of 
(\ref{eq:gendarbp})  by a scalar.
We still assume $B_0=1$ in the computations, and then homogenize the expressions by inserting the powers of 
$B_0$ to match the degrees of monomials.
The shift (\ref{eq:tranl3}) leads to these identification relations between 
the coefficients of  (\ref{eq:gendarbp}) and (\ref{eq:canorthpara}):
\begin{align}\label{eq:sub1}
c_1= & \, \hat c_1+6b_1t_1+2b_2t_2+2b_3t_3, \\
d_1= & \, \hat d_1+2b_2t_3+2b_3t_2,\\
e_1= & \, \hat e_1+b_1(3t_1^2+t_2^2+t_3^2)+2b_2t_1t_2+2b_3t_1t_3+\hat c_1t_1+\hat d_3t_2+\hat d_2t_3,\\
f_0= & \, 2(b_1t_1+b_2t_2+b_3t_3)(t_1^2+t_2^2+t_3^2) + \hat c_1t_1^2 + \hat c_2t_2^2 + \hat c_3t_3^2 \nonumber\\
& + 2\hat d_3t_1t_2 + 2\hat d_2t_1t_3 + 2\hat d_1t_2t_3 + 2\hat e_1t_1 + 2\hat e_2t_2 + 2\hat e_3t_3.\label{eq:sub2}
\end{align}
The expressions for $c_2,c_3,d_2,d_3,e_2,e_3$ are obtained by the symmetries $\sigma_{12},\sigma_{13}$.
Up to the homogenization, the space $\DD_3$ is defined by the ideal generated by these relations 
and the polynomials of Lemma \ref{th:orthcubic}. 
Elimination of the coefficients  $\hat c_1,\dots,\hat d_1,\ldots,\hat e_3$ is straightforward.
The accordingly modified equations $U=0$, $\sigma_{12}U=0$, $\sigma_{13}U=0$
are linear in $t_1,t_2,t_3$ with the discriminant $B_0^2$.
We solve in the non-homogeneous form (i.e., keeping $B_0=1$):
\begin{equation}
t_1= \frac{-b_1c_2-b_1c_3+b_2d_3+b_3d_2+b_1W_3}{2},
\end{equation}
and the respective $\sigma_{12}$, $\sigma_{13}$ modifications for expressions for $t_2,t_3$. 
Now we an express $f_0$ using (\ref{eq:sub2}), and $e_1$, $e_2$, $e_3$ using the last 3 equations in (\ref{eq:cbot10}).

\section{The whole space of Dupin cyclides}
\label{sec:whole}

It is useful 
to compute the projective closure of the variety $\DD_4\subset\RR^{13}$ of quartic Dupin cyclides. 
If this closure contains the variety $\DD_3$ of cubic Dupin cyclides as a component at $a_0=0$,
it is natural to define the whole space $\DD_0$ of Dupin cyclides as this Zariski closure in $\PP^{13}$.
In Section \ref{sc:climit} we indeed conclude that $\DD_3$ is contained in the closure.
As it turns out, the infinite limit $a_0=0$ includes also reducible components with $b_1^2+b_2^2+b_3^2=0$.
We discard the components with complex (rather than all real) points 
\mbox{$(b_1:b_2:\cdots:f_0)\in\PP^{12}\subset\PP^{13}$,}
and describe the quadratic limit surfaces in Remark \ref{rm:quadr}.
The geometric characteristics such as the dimension, the degree and the Hilbert series 
of $\DD_0$, $\DD_4^*$ and $\DD_3$ are presented in  Section \ref{sc:climit1}.

\subsection{Cubic cyclides as limits of quartic cyclides}
\label{sc:climit}

An alternative way to obtain the variety $\DD_3$ of cubic Dupin cyclides is to consider the projective limit $a_0\to 0$ 
of the variety $\DD_4$ of quartic cyclides. The latter variety 
is the restriction $a_0=1$ of the whole space $\DD_0$ of Dupin cyclides.
This projective variety $\DD_0$ is defined by homogenizing the vanishing polynomials for $\DD_4$ with $a_0$.
Taking $a_0=0$ in $\DD_0$ gives a limiting variety that we identify with $\DD_3$ after throwing out complex components.
The general picture of the introduced varieties of Dupin cyclides and the ambient spaces is depicted in Figure \ref{fig:diagram}.
\begin{figure}   
\begin{tikzcd}
	& & \boxed{\PP^{13}} \!\! \arrow[dll,dash] \!\arrow[d,dash] \arrow[drr,dash]& & & \\
	\! \boxed{\PP^{12} \!\, (a_0\!=\!0)} \arrow[dr,dash] \! & & \! \boxed{\DD_0} \! \arrow[dr,dash] \!\arrow[dl,dash] 
	& & \boxed{\RR^{13} \!\, (a_0\!=\!1)} \!\!\arrow[dl,dash] \!\arrow[dr,dash]& \\
	& \boxed{\DD_3} & & \boxed{\DD_4} \!\!\arrow[dr,dash] & & \!\!\boxed{\RR^{10}} \!\arrow[dl,dash] \\
	& & & & \boxed{\DD_4^{*}}& 
\end{tikzcd}
\caption{The inclusion diagram for the varieties 
	of Dupin cyclides embedded in the spaces of Darboux cyclides.}
\label{fig:diagram}
\end{figure}
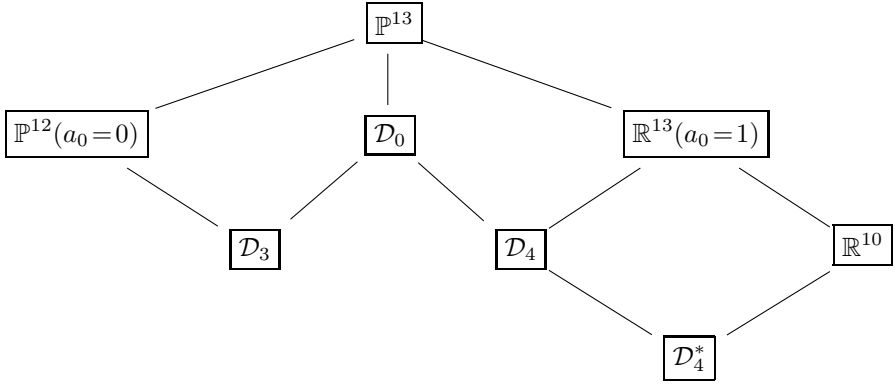

The ideal in $\RR[b_1,b_2,b_3,c_1,c_2,c_3,d_1,d_2,d_3,e_1,e_2,e_3,f_0]$
of the variety $\DD_4$ 
is obtained from our main results  on $\DD^*_4$ by employing 
the normalizing shift (\ref{eq:transl}). 
This shift transforms the general equation (\ref{eq:gendarb}) for Darboux cyclides to
\begin{align} \label{eq:edarb}
\! \big(x^2+y^2+z^2\big)^2 & + \! \left(c_1-b_1^2-\frac{B_0}{2}\right) \! x^2 
+ \! \left(c_2-b_2^2-\frac{B_0}{2}\right) \! y^2 + \! \left(c_3-b_3^2-\frac{B_0}{2}\right) \! z^2 \nonumber \\
& +2\,(d_1-b_{2} b_{3})\,yz+2\,(d_2-b_{1} b_{3})\,xz+2\,(d_3-b_{1} b_{2})\,xy  \nonumber  \\[2pt]
& +2\left(e_1+\frac{b_1\,(B_0-c_1)-b_2d_3-b_3d_2}{2}\right)x  \nonumber \\
& +2\left(e_2+\frac{b_2\,(B_0-c_2)-b_1d_3-b_3d_1}{2}\right)y \\
& +2\left(e_3+\frac{b_3\,(B_0-c_3)-b_1d_2-b_2d_1}{2}\right)z  \nonumber \\
& +f_{0}- \frac{3B_0^2}{16}  + \frac{W_3}{4} - b_1 e_1 - b_2 e_2 - b_3 e_3 \; = \; 0. \nonumber 
\end{align}
Comparing the coefficients here with those in (\ref{eq:mainForm}), 
we modify the equations for the variety $\DD^*_4$ in Proposition \ref{th:d0gen} or Theorem \ref{th:m1}, 
and obtain the defining equations for  the space $\DD_4$. 

Moving towards $\DD_3$ in Figure \ref{fig:diagram}, 
the homogenized ideal for $\DD_0$ is specified using the following standard result.
\begin{propose}
\label{thrm:David-Cox}
Let $I$ be an ideal of the polynomial ring $k[x_1,\dots,x_n]$ over a field $k$, 
and let $\{g_1,\dots,g_t\}$ be a Gr\"obner basis for $I$ with respect to a graded monomial ordering 
in $k[x_1,\dots,x_n]$. Denote by $f^h\in k[x_0,\dots,x_n]$ the homogenization of a polynomial $f\in k[x_1,\dots,x_n]$ 
with respect to the variable $x_0$. Then $\{g_1^h,\dots,g_t^h\}$ is a Gr\"obner basis 
for the homogenized ideal $I^h=\left(f^h\mid f\in I\right)\subset k[x_0,\dots,x_n]$. 
\end{propose}
\begin{proof} 
This is Theorem 4 in \cite[\S 8.4]{IdealsVars}.
\end{proof}

We used {\sf Singular} computations with respect to the total degree monomial ordering \\
$grevlex(b_1,b_2,b_3,...,f_0)$ \cite[pg.~52]{IdealsVars}. 
The computed Gr\"obner basis for $\DD_0$ has 530 elements; 
the computation took about an hour on {\sf Singular}.
After the homogenization with $a_0$ and setting $a_0=0$,
we get a reducible variety, where some components (possibly one)
are restricted by $B_0=0$.
We ignore these components by assuming $B_0=1$ additionally. 
Then the Gr\"obner basis with respect to 
$grevlex(b_1,b_2,b_3,...,f_0)$ has $321$ elements.
Elimination of $e_1,e_2,e_3,f_0$ leads to the expressions of Theorem \ref{th:m2} with $B_0=1$,
without any relation between the other coefficients of (\ref{eq:gendarbp}).
This completes the alternative way of obtaining the ideal for $\DD_3$.

\begin{remark}\rm \label{rm:quadr}
It is interesting to see what quadratic surfaces with $a_0=b_1=b_2=b_3=0$ in (\ref{eq:gendarb})
are contained in 
the variety $\DD_0$ as Dupin cyclides.
After the substitution $a_0=b_1=b_2=b_3=0$ in the Gr\"obner basis with 530 elements, 
we obtain a reducible variety with two components of codimension 2 in $\PP^{9}\subset\PP^{13}$ (over $\CC$). 
One component is defined by vanishing 
of two polynomials: the discriminant of the characteristic polynomial of the matrix $P$ in (\ref{eq:qfmatrices}),
and the determinant of this extended matrix: 
\begin{align}
	\widehat{P} =
	\begin{pmatrix}
		c_1&d_3&d_2&e_1\\
		d_3&c_2&d_1&e_2\\
		d_2&d_1&c_3&e_3\\
		e_1&e_2&e_3&f_0
	\end{pmatrix}.
\end{align} 
The discriminant equals this sum of squares:
\begin{equation}
	S_0^2+S_1^2+\sigma_{12}S_1^2+\sigma_{13}S_1^2
	+15T_1^{\,2}+15\sigma_{12}T_1^{\,2}+15\sigma_{13}T_1^{\,2},
\end{equation}
where
\begin{align*}
	S_0 & = (c_3-c_2)d_1^2+(c_1-c_3)d_2^2+(c_2-c_1)d_3^2+(c_1-c_2)(c_1-c_3)(c_2-c_3), \\
	S_1 & = d_1(d_2^2+d_3^2-2d_1^2)+(c_2+c_3-2c_1)d_2d_3+2(c_2-c_1)(c_3-c_1)d_1, \\
	T_1 & = d_1(d_2^2-d_3^2)+(c_2-c_3)d_2d_3.
\end{align*}
The real points of this component are therefore defined by the ideal generated by the 7 cubic polynomials
$S_0,S_1,\sigma_{12}S_1,\sigma_{13}S_1,T_1,\sigma_{12}T_1,\sigma_{13}T_1$. 
This ideal defines a variety $\DD_2$ of co-dimension 2.
This variety is intersected with the degree four hypersurface $\det\widehat{P} =0$
(which prescribes a singularity on the quadratic surface).

The second component restricts only the coefficients $c_1,c_2,c_3,d_1,d_2,d_3$ and coincides with $\DD_2$.
Hence, it subsumes the real points of the first component. The variety $\DD_2$ specifies the quadratic part of (\ref{eq:gendarb})
to be in the $O(3)$-orbit of \mbox{$A_1x^2+A_1y^2+A_2z^2$}. 
The projective dimension of this orbit is 3
(rather than $4=1+3$) because the rotations around the $z$-axis preserve  $A_1x^2+A_1y^2+A_2z^2$. 
Based on the identification with $\DD_2$, the rotational quadratic surfaces (and the paraboloids like $z^2+x=0$) 
can be considered as Dupin cyclides. 
\end{remark}

\begin{remark} \rm \label{rem:(6:1)} 
We have seen from Theorem \ref{th:m2} that a cubic Dupin cyclide is determined uniquely 
from the coefficients $b_1,b_2,\dots,d_3$ to the cubic and quadratic monomials in (\ref{eq:gendarbp}).
For quartic Dupin cyclides (\ref{eq:gendarb}) with $a_0\neq 0$, the projection $\DD_0\to\PP^9$ 
to the coefficients $a_0,b_1,b_2,\dots,d_3$ is a 6:1  
map generically. It is sufficient to see this 6:1 correspondence for the canonical form (\ref{eq:orthdupin}).
With fixed $A_1,A_2,A_3$, there are these 6 possibilities for the linear part: 
$\pm Dx+F$,  $\pm (\sigma_{12}D)y+\sigma_{12}F$ and $\pm (\sigma_{13}D)z+\sigma_{13}F$, 
with $D,F$ satisfying (\ref{eq:starta})--(\ref{eq:startb}).
It can also be checked (by computing a  Gr\"obner basis) 
that a monomial basis for $\DI_4^*$ in the ring $\RR(c_1,c_2,c_3,d_1,d_2,d_3)[e_1,e_2,e_3,f_0]$ 
has 6 elements.
\end{remark}

\subsection{Dimension, degree, Hilbert series}
\label{sc:climit1}
It is straightforward to compute the Hilbert series for the computed ideals using {\sf Singular} or {\sf Maple}.
The Hilbert series for the projective variety $\DD_0$
is the rational function $H_1(t)/(1-t)^{10}$ with
\begin{align}
	H_1(t)= &\, 1+4t+10t^2+20t^3+35t^4+46t^5+39t^6+10t^7-14t^8-48t^9 \nonumber\\
	& +25t^{10}+56t^{11}-105t^{12}+84t^{13}-37t^{14}+9t^{15}-t^{16}.
\end{align}
The dimension of $\DD_0$ is indeed $10-1=9$, and the degree equals $H_1(1)=134$.
There are ${8\choose 3}-46=10$ linearly independent polynomials of the minimal degree 5.

The Hilbert series for the affine variety $\DD_4^*$ is 
\begin{align}
	\frac{(1+t)(1+t+t^2)(1+2t+4t^2+t^3-t^4)}{(1-t)^6}.
\end{align}
The dimension of $\DD_4^*$ is 6, and the degree equals 42.

The Hilbert series for projective variety $\DD_3$ is 
\begin{align}
	\frac{1+4t+10t^2+20t^3+35t^4-25t^5-9t^6+10t^7}{(1-t)^9}.
\end{align}
The dimension of $\DD_3$ is $9-1=8$, and the degree equals 46. 
Recall that $\DD_3$ describes the component $B_0\neq 0$ of the subvariety $a_0=0$ of $\DD_0$.
The homogeneous version of the Gr\"obner basis for $\DD_3$ 
with respect to $grevlex(b_1,b_2,b_3,...,f_0)$ has 261 elements. 
It has ${8\choose 3}+25=81$ linearly independent polynomials of the minimal degree 5.

The co-dimension of all considered spaces of Dupin cyclides 
inside the corresponding projective spaces of Darboux cyclides equals 4.

\begin{remark} \rm
	Beside the usual grading, the equations of our varieties have a weighted degree
	that reflects their invariance under the scaling $(x,y,z)\mapsto (\lambda x,\lambda y,\lambda z)$ with $\lambda\in\RR$.
	The weights of the variables are the following: 
	\[ 
	wd(b_j)=1, \quad wd(c_j)=wd(d_j)=2, \quad wd(e_j)=3 \quad \mbox{for} \quad j\in\{1,2,3\},
	\] 
	and  $wd(a_0)=0$,  $wd(f_0)=4$.
	The symmetries $\sigma_{12},\sigma_{13}$ do not change the weighted degree.
	The 12 equations in  Proposition \ref{th:d0gen} 
	have these weighted degrees:
	\[ 
	wd(K_j)=8, \qquad wd(L_j)=7, \qquad wd(M_k)=9 \qquad \mbox{for} \quad j\in\{1,2,3\},
	\] and $wd(N_1)=8$, $wd(N_2)=10$, $wd(N_3)=12$.
\end{remark}

\section{Classification of real cases of Dupin cyclides}
\label{sec:class}

The torus equation (\ref{eq:torus}) is defined over $\RR$ also when $r^2<0$ or $R^2<0$.
Similarly, the canonical equation  (\ref{eq:dupin}) is defined over $\RR$ if $\alpha,\beta,\gamma\in\RR$,
or if exactly two of these numbers 
are on the imaginary line $\sqrt{-1}\,\RR\subset\CC$. 
Then we may obtain degenerations to surfaces with a few (if any) real points.

Section \ref{sc:degenerate} classifies all degenerations of Dupin cyclides.
For that purpose, Section \ref{sec:izomobius} defines general M\"obius isomorphisms between Dupin cyclides and toruses,
and follows the cases when they are defined over $\RR$. 
Section \ref{sc:qrr} defines a M\"obius invariant $J_0$ of Dupin cyclides.
This invariant and a few semi-algebraic conditions classify the Dupin cyclides up to real M\"obius transformations.

\subsection{M\"obius isomorphisms to the torus}
\label{sec:izomobius}


Spherical inversions or M\"obius transformations between Dupin cyclides and a general torus
have been constructed geometrically \cite[Theorem 3.5]{Ottens} and computed \cite[\S 2.3]{vdValk}.
An explicit M\"obius isomorphism that maps a canonical Dupin cyclide (\ref{eq:dupin}) 
to the torus equation (\ref{eq:torus}) is given by
\begin{align} \label{eq:mobt}
	(x,y,z)\mapsto & \frac{\alpha\delta+\beta\varepsilon}{\gamma}\,(1,0,0)
	+ \frac{2\beta\varepsilon}{(x-\gamma)^2+y^2+z^2}\, (x-\gamma,y,z),
\end{align}
where $\beta=\sqrt{\alpha^2-\gamma^2}$, $\varepsilon=\sqrt{\delta^2-\gamma^2}$. 
This M\"obius transformation is defined over $\RR$ if and only if $\gamma\in\RR\setminus\{0\}$ and $\beta\varepsilon\in\RR$.
If $\gamma=0$, 
the canonical equation (\ref{eq:dupin}) coincides already with the torus equation (\ref{eq:torus})  
with  $\alpha=R$, $\delta=r$. With $\beta\varepsilon\in\RR$, the minor and major radiuses of the torus are given by
\begin{equation}  \label{eq:rrr}
	r=\frac{\gamma^2\varepsilon}{\alpha\varepsilon+\beta\delta}, \qquad 
	R=\frac{\gamma^2\beta}{\alpha\varepsilon+\beta\delta},
\end{equation}
We can further apply scaling by 
the factor 
\begin{equation} \label{eq:scale}
	\frac{\gamma^2}{\alpha\varepsilon+\beta\delta}
\end{equation}
to the immediate torus equation 
if either all $\alpha,\delta,\beta,\varepsilon\in\RR$ or all 
$\alpha,\delta,\beta,\varepsilon\in\sqrt{-1}\,\RR$.
The resulting torus equation has $r=\varepsilon$,  $R=\beta$. 
Otherwise the scaling can be adjusted by the factor $\sqrt{-1}$, 
and the radiuses become 
$r^2=\gamma^2-\delta^2$, $R^2=\gamma^2-\alpha^2$.

On the other hand, the canonical equation (\ref{eq:dupin}) 
is symmetrical \cite[(1)-(2)]{DuPratt1}
with respect to the simultaneous interchange $y\leftrightarrow z$, $\alpha\leftrightarrow\gamma$. 
This symmetry implies a M\"obius equivalence to the torus with the minor radius $\sqrt{R^2-r^2}$ 
(and the same major radius $R$) as well. 
Up to scaling, this M\"obius equivalence is symmetric to (\ref{eq:mobt}):
\begin{align} \label{eq:mobt2a}
	(x,y,z)\mapsto & \; \frac{\gamma\delta+\sqrt{(\gamma^2-\alpha^2)(\delta^2-\alpha^2)}}{\alpha}\,(1,0,0) \nonumber\\[2pt]
	& \; + \frac{2\sqrt{(\gamma^2-\alpha^2)(\delta^2-\alpha^2)}}{(x-\alpha)^2+y^2+z^2}\, (x-\alpha,z,y),
\end{align}
Both M\"obius isomorphisms are defined over $\RR$ if $\alpha\gamma\neq 0$ and  
$\delta^2$ is between $\alpha^2$ and $\gamma^2$ on the real line.
This means that  $0<r/R<1$, because (\ref{eq:rrr}) implies  
\begin{equation} \label{eq:rRt}
	\frac{r}{R} = \sqrt{\frac{\delta^2-\gamma^2}{\alpha^2-\gamma^2}}.
\end{equation}
A composition of the two M\"obius isomorphisms relates 
two toruses with the minor radiuses $r$ and $\sqrt{R^2-r^2}$, when $r<R$.
This transformation is obtained explicitly by applying (\ref{eq:mobt2a}) with
$\alpha=R$, $\gamma=0$, $\delta=r$. After additional scaling of $(x,y,z)$ by $r/R$ we obtain 
the M\"obius transformation
\begin{align} \label{eq:mobt2}
	(x,y,z)\mapsto & \big(\sqrt{R^2-r^2},0,0\big)
	+ \frac{2r\,\sqrt{R^2-r^2}}{(x-r)^2+y^2+z^2}\, (x-r,z,y),
\end{align}
that brings torus (\ref{eq:torus}) to the torus 
\begin{equation} \label{eq:torus2}
	(x^2+y^2+z^2+r^2)^2-4R^2(x^2+y^2)=0.
\end{equation}
When $r^2\ge R^2$, this M\"obius duality is not defined over $\RR$. 
If $0<R^2<r^2$, then (\ref{eq:torus}) defines a singular {\em spindle} torus \cite[p.~288]{DuChand}, 
and the surface (\ref{eq:torus2}) has no real points. 
If $0<R^2=r^2$, then (\ref{eq:torus}) defines a singular {\em horn} torus, 
while (\ref{eq:torus2}) defines a circle.

\subsection{The toroidic invariant for $r/R$}
\label{sc:qrr}

Any torus (\ref{eq:torus}) has two clear families of circles on it, namely on the vertical planes $ax+by=0$
or horizontal planes $z=c$. These circles are the principal curvature lines on the torus, and are known as
{\em principal circles}. Less known are two families of {\em Villarceau circles} \cite{VilCircles} 
on the bitangent planes $z=ax+by$ with $a^2+b^2=r^2/(R^2-r^2)$, on smooth toruses with $r<R$. 
The angles $\theta$, $\frac{\pi}2-\theta$ between principal and Villarceau circles depend only on the families of the involved circles.
The sine (or the complementary cosine) of $\theta$ 
equals \cite{VilCircles} to the quotient $r/R$.
The duality (\ref{eq:mobt2}) of toruses with 
the minor radiuses $r$ and $\sqrt{R^2-r^2}$
underscores constancy of the angle pair $\theta$, $\frac{\pi}2-\theta$  
under the (conformal!) M\"obius transformations. 
Krasauskas noted to us that the numbers 
$r^2/R^2$ and $1-r^2/R^2$ are equal to the two possible cross-ratios 
within the quaternionic representation \cite{Zube} of Dupin cyclides.

The symmetry between 
$r/R$ and $\sqrt{1-r^2/R^2}$ 
leads to this invariant under the M\"obius transformations:
\begin{equation}
	J_0 =\frac{r^2}{R^2} \left(1-\frac{r^2}{R^2}\right).
\end{equation}
We define the invariant $J_0$ for general Dupin cyclides by the M\"obius equivalence. 
The maximal value $J_0=1/4$ gives the ``most round" cyclides (with $R=\sqrt{2}\,r$) 
that optimize the Willmore energy \cite{White73}
\begin{equation}
	\int\!\!\int_{S\,} H^2 dA
\end{equation}
for the smooth real surfaces $S$ with the torus topology. 
The integrand $H^2$ is the mean curvature $H$ squared,
and $dA$ is the infinitesimal area element.
The Willmore energy is conformally invariant,  
and it equals  \cite[pg 275]{Willmore}
\begin{equation}
	\frac{\pi^2R^2}{r\,\sqrt{R^2-r^2}}  = \frac{\pi^2}{\sqrt{J_0}}
\end{equation}
for a smooth torus (\ref{eq:torus}).
The duality breaks down for $J_0\le 0$, as M\"obius transformation (\ref{eq:mobt2}) is then not defined over $\RR$.
The singular horn torus (with $r=R$) is paired to the degeneration to a circle ($r=0$), 
and the spindle toruses (with $r>R$) are paired with algebraic surfaces with no real points ($r\in \sqrt{-1}\,\RR$).

We apply our results to compute the invariant $J_0$ for all Dupin cyclides.
In the case of a canonical quartic cyclide (\ref{eq:orthdupin}), 
we choose M\"obius equivalence (\ref{eq:mobt}) with a torus,
and obtain
\begin{align}
	\frac{r^2}{R^2}= \frac{\delta^2-\gamma^2}{\alpha^2-\gamma^2}=\frac{2A_2+A_3-A_1}{A_2-A_3}
\end{align}
due to (\ref{eq:rRt}) and (\ref{eq:greeksq}). Then
\begin{align} \label{eq:torcos}
	J_0 & =-\frac{(\delta^2-\gamma^2)(\delta^2-\alpha^2)}{(\alpha^2-\gamma^2)^2} \\
	& = \frac{(A_1-2A_2-A_3)(A_2+2A_3-A_1)}{(A_2-A_3)^2}.
\end{align}
To obtain expressions of $J_0$ for the general quartic cyclide (\ref{eq:mainForm}),
we first eliminate $A_2,A_3$ 
as in Lemma \ref{th:dgraph}. The result is
\begin{align} \label{eq:j0witha}
	J_0 = \frac{7A_1^2-8C_0A_1+2C_0^2+W_1}{3A_1^2-2C_0A_1-C_0^2+4W_1}.
\end{align}
After the elimination of $A_1$, we generically obtain
\begin{align}
	J_0 & = \frac{6C_0(d_2e_1+d_1e_2-c_1e_3-c_2e_3)+(C_0^2+8W_1+28f_0)e_3}
	{4C_0(d_2e_1+d_1e_2-c_1e_3-c_2e_3)+(7W_1+12f_0)e_3} \\
	& = \frac{(4f_0-C_0^2)(28f_0+C_0^2)+4(8f_0+C_0^2)W_1-12C_0(W_2-2E_0)}
	{12f_0(4f_0-C_0^2)+(28f_0+C_0^2)W_1-8C_0(W_2-2E_0)}.
\end{align}
These expressions were obtained after heavy Gr\"obner basis computations with the new variable $J_0$
(or rather, its numerator and denominator separately) and the superfluous $A_1$. 
They can be checked by reducing (to $0$) the numerator of the difference to (\ref{eq:j0witha})
in the Gr\"obner basis in $\DR_4^*[A_1]$ for Lemma \ref{th:dgraph}.

To obtain an expression of $J_0$ for the cubic cyclide (\ref{eq:canpara}), 
we apply the procedure at the beginning of this section: 
transform the variables according to the shift (\ref{eq:transl}) and the form (\ref{eq:gendarb}),
homogenize with $a_0$, and set $a_0=0$. 
Here is a relatively compact rational expression of the lowest weighted degree obtained after heavy computations:
\begin{equation} \label{eq:j0cubic}
	J_0 = 3\,\frac{B_0(-2Y_5+c_1c_2+c_1c_3+c_2c_3)+Y_6}{B_0(Y_5+2c_1^2+2c_2^2+2c_3^2+c_1c_2+c_1c_3+c_2c_3)+2Y_6},
\end{equation}
with $Y_5 = d_1^2+d_2^2+d_3^2-4b_1e_1-4b_2e_2-4b_3e_3$, and
\begin{align*}
	Y_6 = & \; 5b_1^2d_1^2+5b_2^2d_2^2+5b_3^2d_3^2 +10b_1b_2(c_3d_3-d_1d_2)+10b_1b_3(c_2d_2-d_1d_3) \nonumber \\
	&  +10b_2b_3(c_1d_1-d_2d_3)-2C_0(b_1b_2d_3+b_1b_3d_2+b_2b_3d_1)  \nonumber \\
	&  -b_1^2(c_1^2+4c_2c_3)-b_2^2(c_2^2+4c_1c_3)-b_3^2(c_3^2+4c_1c_2).
\end{align*}
The variables $e_1,e_2,e_3$ could be eliminated in (\ref{eq:j0cubic}) using Lemma \ref{th:m2}, 
but the obtained rational expression of degree 6 (after using $B_0=1$) is much larger.

\subsection{Classification of degenerate Dupin cyclides}
\label{sc:degenerate}

The cubic Dupin cyclides (\ref{eq:canpara}) always have real points and are easy to classify, 
starting from \cite[p.~151]{DuPratt2}, \cite[p.~288]{DuChand}:
\begin{itemize}
	\item smooth cyclides when $pq<0$;
	\item spindle cyclides when $pq>0$, $p\neq q$;
	\item horn cyclides when $pq=0$, $p\neq q$;
	\item reducible surface (a sphere and a tangent plane) when $p=q\neq0$; 
	\item reducible surface (a plane and a point on it) when $p=q=0$.
\end{itemize}
Degeneration to quadratic surfaces is explained in Remark \ref{rm:quadr}. 

The full classification of the real points defined by the quartic canonical equation (\ref{eq:dupin})
depends on the order of $0,\alpha^2,\gamma^2,\delta^2$ on the real line, as we demonstrate shortly.
The classification is depicted in Figure \ref{fig:degenDu}\,\refpart{a}. 
The border conditions $\alpha^2=0$, $\gamma^2=0$, $\delta^2=0$ are depicted by 3 intersecting circles
(marked on the outer side by $\alpha^2,\gamma^2,\delta^2$). Their inside disks represent positive values 
of $\alpha^2,\gamma^2,\delta^2$, respectively, while the negative values are represented by the outer sides of the circles. 
The conditions $\alpha^2=\gamma^2,\alpha^2=\delta^2,\gamma^2=\delta^2$ are represented by the 3 lines intersecting at the center. 
Their markings near the edge of the picture (say, $\alpha^2$ and $\gamma^2$) indicate which of the values ($\alpha^2$ or $\gamma^2$)
is larger on either side of the line. Most triangular regions are marked by a sequence of inequalities between $0,\alpha^2,\gamma^2,\delta^2$;
these {\em admisible} regions represent the cases when the canonical equation is defined over $\RR$. 
The 6 asymptotic outside regions represent the cases when all three $\alpha^2,\gamma^2,\delta^2$ are negative;  they 
are not admissible. The admissible regions, several edges and vertices are labeled by abbreviations for various types of Dupin cyclides.
The labels are not repeated when the type does not change when passing from a triangular region to its edge or vertex. 
These coincidences are represented by the non-strict inequalities between $0$ and $\alpha^2,\gamma^2,\delta^2$.  
When a distinct edge and its vertex have the same type, the type is indicated near the vertex, and an arrow is added to the direction of that edge.
The values of $J_0$ can be considered as constant along the radial directions from the center $Q$, with the optimal $J_0=1/4$ in the vertical directions,
$J_0=-\infty$ on the horizontal line $\alpha^2=\gamma^2$, and $J_0=0$ along the other two drawn lines; see (\ref{eq:torcos}). 
The trivial case  $(\alpha^2,\gamma^2,\delta^2)=(0,0,0)$ is not represented in  Figure \ref{fig:degenDu}\refpart{a}, 
though it is represented 
by the origin point in Figure \ref{fig:degenDu}\refpart{b}.

\begin{figure}
	\begin{tikzpicture}[scale=0.55]
		\node at (-13,10) {$(a)$};
		\draw[thick] (-2.4641,0) circle [radius=8];\draw[thick] (1,2) circle [radius=8];\draw[thick] (1,-2) circle [radius=8];
		\node at (1.8,10.45) {$\alpha^2$};  \node at (7.7,7.3) {$\alpha^2$};
		\node at (7.3,-7.6) {$\gamma^2$};  \node at (2,-10.4) {$\gamma^2$};
		\node at (-10.1,-3.5) {$\delta^2$};  \node at (-10.1,3.5) {$\delta^2$};
		\draw[very thick] (-11.6,0) -- (10.2,0); 
		\node at (-11.15,0.4) {$\alpha^2$}; \node at (-11.15,-0.5) {$\gamma^2$};
		\node at (9.8,0.4) {$\alpha^2$}; \node at (9.8,-0.5) {$\gamma^2$};
		\draw[thick] (-5.85,9.9) -- (5.7,-10.2);
		\node at (-5.1,9.7) {$\alpha^2$}; \node at (-6,9.2) {$\delta^2$};
		\node at (5.1,-10) {$\delta^2$}; \node at (5.85,-9.4) {$\alpha^2$};
		\draw[thick] (-5.9,-10) -- (5.7,10.2);
		\node at (-5.15,-9.5) {$\gamma^2$}; \node at (-5.85,-9.2) {$\delta^2$}; 
		\node at (5.15,9.9) {$\delta^2$};  \node at (5.8,9.5) {$\gamma^2$};
		\node  at (1.8,8.3) {${\gamma^2\!<\!\delta^2\!\leq\!0\!\leq\!\alpha^2}$};
		\node[rotate=45]  at (-8.4,2.9) { ${\gamma^2\!<\!\alpha^2\!\leq\!0\!\leq\!\delta^2}$};
		\node[rotate=-50]  at (-8.7,-2.2) { ${\alpha^2\!<\!\gamma^2\!\leq\!0\!\leq\!\delta^2\!}$};
		\node at (1.8,-8.5) { ${\alpha^2\!<\!\delta^2\!\leq\!0\!\leq\!\gamma^2\!}$};
		\node[rotate=45]  at (6.4,-5.9) { ${\delta^2\!<\!\alpha^2\!\leq\!0\!\leq\!\gamma^2}$};
		\node[rotate=-45]  at (6.4,5.9) { ${\delta^2\!<\!\gamma^2\!\leq\!0\!\leq\!\alpha^2}$};
		\node  at (0.25,4.5) { ${\!0\!\leq\!\gamma^2\!<\!\delta^2\!<\!\alpha^2}$};
		\node  at (-3.6,1.5) { ${0\!\leq\!\gamma^2\!<\!\alpha^2\!<\!\delta^2}$};
		\node  at (-3.6,-1.5) { ${0\!\leq\!\alpha^2\!<\!\gamma^2\!<\!\delta^2}$};
		\node  at (0.05,-4.36) { ${0\!\leq\!\alpha^2\!<\!\delta^2\!\!<\!\!\gamma^2}$};
		\node  at (3.1,-1.5) { ${0\!\leq\!\delta^2\!<\!\alpha^2\!<\!\gamma^2}$};
		\node  at (3.05,1.4) { ${0\!\leq\!\delta^2\!<\!\gamma^2\!<\!\alpha^2}$};
		\node at (-0.1,2.8) {SM}; 
		\node at (-0.1,-2.8) {SM};
		\node at (3.6,3.2) {SP}; \node at (3.6,-3.2) {SP};\node at (-3.6,3.2) {SP}; \node at (-3.6,-3.2) {SP};
		\node at (-2.45,-4.8) {H}; \draw[-stealth] (-2.5,-4.5)--(-2.1,-3.8); 	\node at (2,3) {H};
		\node at (-2.45,4.8)  {H}; \draw[-stealth] (-2.5,4.5)--(-2.1,3.8); \node at (2,-3) {H};
		\node at (5.15,0.33) {R}; \draw[-stealth] (4.85,0.3)--(4.15,0.3); \node at (-3.25,0.33) {R}; 
		\node at (-0.15,-0.8) {Q};
		\node at (-7.175,0.33) {D};
		\node at (8,3.5) {NP}; \node at (8,-3.5) {NP}; 
		\node at (3.65,5.95) {C}; \node at (3.65,-5.95) {C};  
		\node at (-10.1,0.4) {P}; \draw[-stealth] (-9.9,0.3)-- (-9.2,0.3) ;
		\node at (5.1,8.4) {P};  \draw[-stealth] (4.8,8.15)-- (4.4,7.5) ;
		\node at (5.15,-8.5) {P}; \draw[-stealth] (4.85,-8.2)-- (4.45,-7.55) ;
		\node at (-7,5.5) {PP}; \node at (-1.5,8.75) {PP}; \node at (-7,-5.5) {PP}; \node at (-1.5,-8.75) {PP}; 
	\end{tikzpicture}
	\vspace{0.5cm}
	\begin{multicols}{2}
		\begin{tikzpicture}[scale=1.2]
			\node at (-2,2.5) {$(b)$};
			\draw[->,very thick] (0,0)--(3,0)  node[below]{$r^2$};
			\draw[dashed,thin] (-2,0)--(0,0);
			\draw (0,0)--(0,-0.07);
			\draw[->,very thick] (0,0)--(0,2.5) node[left]{$R^2$};
			\draw[dashed,thin] (0,-1.7)--(0,-0.35);
			\draw[very thick] (0,0)--(2.5,2.5); \node[rotate=45] at (1.8,2.1){$r^2=R^2$};
			\draw[very thick,opacity=0.5] (-1.4,-1.4)--(0,0);
			\node  at (0.6,1.5) {SM};
			\node at (1.85,1.55) {H};
			\node at (1.5,0.7) {SP};
			\node  at (1.5,-0.17) {D};
			\node  at (-1,1) {NP}; 	\node  at (-1.3,-0.5) {NP};
			\node  at (0.05,-0.2) {P}; \draw[-stealth] (-0.1,-0.23) -- (-0.37,-0.5);
			\node  at (1,-1) {PP}; \node  at (-0.5,-1.3) {PP};
			\node at (-0.15,1.3) {C};
		\end{tikzpicture}
		
		$(c)$
		\begin{enumerate}[]
			\item SM: Smooth Dupin cyclide
			\vspace{-0.25cm}
			\item SP:\; Spindle Dupin cyclide
			\vspace{-0.25cm}
			\item H:\;\;\, Horn Dupin cyclide
			\vspace{-0.25cm}
			\item R:\;\;\, Two spheres touching each other
			\vspace{-0.25cm}
			\item Q:\;\;\,  Sphere and one point on it
			\vspace{-0.25cm}
			\item D:\;\;\, Double sphere
			\vspace{-0.25cm}
			\item C:\;\;\, Circle
			\vspace{-0.25cm}
			\item P:\;\;\,  One real point
			\vspace{-0.25cm}
			\item PP: Two real points
			\vspace{-0.25cm}
			\item NP: No real points 
		\end{enumerate}
	\end{multicols}
	\vspace{-0.5cm}
	\caption{\refpart{a} Classification of real points on quartic cyclides in the canonical form  (\ref{eq:dupin}).\\
		\refpart{b} Classification of real points on toruses (\ref{eq:torus}). \refpart{c} Legend. }
	\label{fig:degenDu}
\end{figure}

This classification can be proved from the easier classification of torus equations (\ref{eq:torus}) in Figure \ref{fig:degenDu}\,\refpart{b},
and by considering which of the two M\"obius transformations (\ref{eq:mobt}) and (\ref{eq:mobt2a}) are defined over $\RR$. 
The case $r^2<\min(0,R^2)$ of ``toruses" with no real points can be seen from this alternative form of (\ref{eq:torus}):
\begin{equation}
	(x^2+y^2+z^2-R^2+r^2)^2-4r^2(x^2+y^2)+4(R^2-r^2)z^2=0.
\end{equation}
The cases of toruses lie on the circles $\gamma^2=0$ and $\alpha^2=0$ of Figure \ref{fig:degenDu} \refpart{a}.
The M\"obius equivalence (\ref{eq:mobt}) is defined over $\RR$ if $\gamma^2>0$ and $\gamma^2$ is either larger or smaller 
than both $\alpha^2,\delta^2$. It preserves $J_0$ and maps the applicable triangular regions onto the segments of the circle $\gamma^2=0$ representing toruses.
Adjacency of the corresponding triangular regions and segments cannot hold for the  two lower-right regions SM, SP 
with $\alpha^2,\delta^2\in[0,\gamma^2)$; these are the cases when the scaling by (\ref{eq:scale}) has to be adjusted by $\sqrt{-1}$.
Similarly, (\ref{eq:mobt2a}) is defined over $\RR$ if $\alpha^2>0$ and $\alpha^2$ is either larger or smaller 
than both $\gamma^2,\delta^2$. This covers all cases except the line $\alpha^2=\gamma^2$ and the two leftmost triangular regions.
When $\alpha^2=\gamma^2$, canonical equation (\ref{eq:dupin}) factorizes
and defines (generically) two spheres with the centers at $(x,y,z)=(\pm\alpha,0,0)$ and touching at the point $(\delta,0,0)$. 
If then $\alpha^2<0$, only the touching point is real; the other few deeper degenerations are straightforward.
For the leftmost triangular region with $\gamma^2<\alpha^2\le0$, the canonical equation can be rearranged to
\begin{equation}
	\big(x^2+y^2+z^2-\alpha^2+\gamma^2-\delta^2\big)^2+4(-\gamma^2+\alpha^2)z^2+(-4)(\alpha\delta-\gamma x)^2=0.
\end{equation} 
All three terms are positive for that region, and the Dupin surface then consists of two points on the line $z=0$, $x=\alpha\delta/\gamma$.
Similarly, two points are obtained for the other leftmost region $\alpha^2<\gamma^2\le0$.

The conditions on $\alpha^2,\gamma^2,\delta^2$ can be directly translated to the conditions on the coefficients $A_1,A_2,A_3$ in (\ref{eq:orthdupin})
using (\ref{eq:greeksq}). The quadratic covering (\ref{eq:starta}) of the $(A_1,A_2,A_3)$-plane confirms the topology of 4 admissible regions 
connected at 6 corners. 
The translated classification is as follows:
\begin{itemize}
	\item Smooth Dupin cyclides, when $A_1\le\min(A_2,A_3)<A_1-A_2-A_3<\max(A_2,A_3)$.
	Then $J_0\in(0,\frac14]$ and $A_2\neq A_3$, $A_2+A_3<0$.
	\item Horn cyclides, when either $A_1-A_2-A_3=\max(A_2,A_3)<0$, $A_2\neq A_3$,
	or $A_1-\min(A_2,A_3)=A_2+A_3<0$, $A_2\neq A_3$. In either case, $J_0=0$.
	\item Spindle cyclides, when either $A_1\le \min(A_2,A_3)<\max(A_2,A_3)<A_1-A_2-A_3$. 
	or $A_1-\min(A_2,A_3)<A_2+A_3\le 0$, $A_2\neq A_3$. In either case, $J_0<0$, $A_2\neq A_3$.
	\item Reducible surface of two touching spheres, when 
	$3A_2<A_1<A_2=A_3$  (then $A_2=A_3<0$)
	or $A_1<A_2=A_3\le 0$. In either case, $J_0=-\infty$.
	\item Reducible surface of a sphere and a point on it, when $A_2=A_3=\frac13A_1<0$.  Then $J_0$ is undefined.
	\item Double sphere, when $A_1=A_2=A_3<0$. Then $J_0=-\infty$.
	\item A circle, when $A_2+A_3=0$, $A_1=\min(A_2,A_3)<0$. Then $J_0=0$, $A_1<0$. 
	\item Two real points, when either $\min(A_2,A_3)<A_1-A_2-A_3\le A_1\le\max(A_2,A_3)$
	(then $J_0>0$, $A_2\neq A_3$, $A_2+A_3\ge 0$)  or $\max(A_2,A_3)\le A_1\le A_1-A_2-A_3$, $A_2\neq A_3$
	(then $J_0<0$, $A_2+A_3\le 0$).
	\item One real point, when either $A_1-A_2-A_3=\min(A_2,A_3)\le 0$, $A_2+A_3>0$ (then $J_0=0$), or 
	$0\le A_2=A_3< A_1$ (then $J_0=-\infty$), or $A_1=A_2=A_3=0$.
	\item No real points, when $A_1-A_2-A_3<\min(A_2,A_3)\le A_1\le\max(A_2,A_3)$. Then $J_0<0$, $A_2+A_3>0$.
\end{itemize}
Further translation in terms of the coefficients in (\ref{eq:mainForm}) is cumbersome.  
Some basic distinctions are determined by the $J_0$-invariant, represented by the directions from the central point Q in Figure \ref{fig:degenDu}\refpart{a}.
The circular boundaries $\alpha^2=0$, $\gamma^2=0$, $\delta^2=0$ do not represent semi-algebraic conditions (except at the vertices C, D), 
as they separate the cases of whether the surface equations (\ref{eq:dupin}) and eventually (\ref{eq:mainForm}), (\ref{eq:gendarb}) are defined over $\RR$ or not.
To distinguish the 6 regions around Q and the 6 outer regions, it is tempting to invent a polynomial that vanishes at the vertices C, D 
and has different signs for the inner and outer regions. But the polynomials in $\alpha^2,\gamma^2,\delta^2$ (or the other coefficients) that vanish at C, D
also vanish at the respectively opposite meeting corners or the PP and NP regions. 
The practical suggestion to distinguish the cases is to compute $A_1$ using one of the equations (linear in $A_1$) of Lemma \ref{th:dgraph},
and then compute $A_2,A_3$ as the roots of the quadratic polynomial $X^2-(A_2+A_3)X+A_2A_3$. 
After eliminating $A_2,A_3$, we get the equation
\begin{equation} \label{eq:xa1a3}
	X^2+(A_1-C_0)X+W_1-C_0A_1+A_1^2=0
\end{equation}
with the roots $X=A_2$, $X=A_3$. If preferable, one can reduce the degree in $A_1$ in the last equation using (\ref{eq:eqqf4}).

\section*{Acknowledgments}
The authors would like to thank Rimvydas Krasauskas and Severinas Zub\.e for useful remarks and suggestions.

\section*{Funding}
This work is part of a project that has received funding from the European Union’s Horizon 2020 research and innovation programme under the Marie Skłodowska-Curie grant agreement No 860843.

\end{document}